\documentclass[11pt,a4paper]{amsart}

\usepackage[backend=biber, style=alphabetic, maxbibnames=99, safeinputenc]{biblatex}
\DeclareFieldFormat{postnote}{#1}
\DeclareFieldFormat{multipostnote}{#1} 
\usepackage{mymacros_english_paper}
\usepackage{mymacros_common}

\begin{document}
\title[Tilting and untilting for ideals in perfectoid rings]{Tilting and untilting for ideals in perfectoid rings}

\author[D. Dine]{Dimitri Dine}
\address{Wiesbaden, Germany}
\email{dimas.dine@gmail.com}

\author[R. Ishizuka]{Ryo Ishizuka}
\address{Department of Mathematics, Tokyo Institute of Technology, 2-12-1 Ookayama, Meguro, Tokyo 152-8551}
\email{ishizuka.r.ac@m.titech.ac.jp}




\begin{abstract}
    For a perfectoid ring $R$ of characteristic $0$ with tilt $R^{\flat}$, we introduce and study a tilting map $(-)^{\flat}$ from the set of $p$-adically closed ideals of $R$ to the set of ideals of $R^{\flat}$ and an untilting map $(-)^{\sharp}$ from the set of radical ideals of $R^{\flat}$ to the set of ideals of $R$. The untilting map $(-)^{\sharp}$ is defined purely algebraically and generalizes the analytically defined untilting map on closed radical ideals of a perfectoid Tate ring of characteristic $p$ introduced in \cite{dine2022Topologicala}.
    We prove that the two maps
    \begin{equation*}
        J\mapsto J^{\flat}~\textrm{and}~I\mapsto I^{\sharp}
    \end{equation*}
    define an inclusion-preserving bijection between the set of ideals $J$ of $R$ such that the quotient $R/J$ is perfectoid and the set of $p^{\flat}$-adically closed radical ideals of $R^{\flat}$, where $p^{\flat}\in R^{\flat}$ corresponds to a compatible system of $p$-power roots of a unit multiple of $p$ in $R$.

    Finally, we prove that the maps $(-)^{\flat}$ and $(-)^{\sharp}$ send (closed) prime ideals to prime ideals and thus define a homeomorphism between the subspace of $\Spec(R)$ consisting of prime ideals $\mathfrak{p}$ of $R$ such that $R/\mathfrak{p}$ is perfectoid and the subspace of $\Spec(R^{\flat})$ consisting of $p^{\flat}$-adically closed prime ideals of $R^{\flat}$. In particular, we obtain a generalization and a new proof of the main result of \cite{dine2022Topologicala} which concerned prime ideals in perfectoid Tate rings.
\end{abstract}

\maketitle 

\setcounter{tocdepth}{1}
\tableofcontents

\section{Introduction}

Throughout this paper, we fix a prime number $p$. 

Since their first appearance in Scholze's thesis \cite{scholze2012Perfectoida}, perfectoid algebras and spaces have revolutionized both $p$-adic geometry and mixedcharacteristic commutative algebra, leading, among many other applications, to Andr\'e's proof of the direct summand conjecture \cite{andre2018Conjecture, andre2018Lemme}.
In the realm of commutative algebra, there is also the notion of (integral) perfectoid rings, introduced in \cite[Definition 3.5]{bhatt2018Integral}.
The class of (integral) perfectoid rings includes not only rings of power-bounded elements of perfectoid algebras in the sense of \cite{scholze2012Perfectoida} (and, more generally, of perfectoid Tate rings in the sense of \cite{fontaine2013Perfectoides}) but also contains some rings which need not even be $p$-torsion-free. 

In \cite{dine2022Topologicala}, one of the authors proved that, for a perfectoid Tate ring $A$ with tilt $A^{\flat}$, there is a homeomorphism between the space of so-called \textit{spectrally reduced prime ideals} of the Tate ring $A$ (by a theorem of Bhatt and Scholze, these coincide with the prime ideals $\mathfrak{p}$ of $A$ such that the quotient $A/\mathfrak{p}$ is again a perfectoid Tate ring, see \cite[Theorem 4.4]{dine2022Topologicala}) and the space of closed radical ideals of $A^{\flat}$, the two spaces being endowed with their respective Zariski topologies (\cite[Theorem 4.16]{dine2022Topologicala}).
In particular, a perfectoid Tate ring is an integral domain if and only if its tilt is an integral domain.
In this paper, we use algebraic techniques (in particular, a purely algebraic definition of the untilting operation, see \Cref{TiltUntiltIdeal}) to generalize the main notions and results of \cite{dine2022Topologicala} from the case of perfectoid Tate rings to perfectoid rings in the sense of \cite{bhatt2018Integral}.

In this paper, we use perfectoid rings in the sense of \cite{bhatt2018Integral}.
For any perfectoid ring $R$ of characteristic $0$, we fix some unit multiple $vp$ of $p$ which admits a compatible system of $p$-power roots of $vp$ in $R$ (by \cite[Lemma 3.9]{bhatt2018Integral}, such a unit multiple of $p$ always exists) and we fix an element $p^{\flat}\in R^{\flat}$ of the tilt $R^{\flat} \cong \varprojlim_{F} R$ of $R$ corresponding to a compatible system of $p$-power roots of $vp$ (it is adapted to the notation used in \cite{cesnavicius2023Purity}. Note that a compatible system of \(p\)-power roots of \(p\) does not necessarily exist).

Then our main result can be stated as follows.
The tilting operation \((-)^\flat\) for rings and ideals is defined in \Cref{TiltRing} and \Cref{TiltUntiltIdeal}.

\begin{maintheorema}[{\Cref{CorrespPerfdIdeals} and \Cref{CorrespPrimePerfdIdeals}}] \label{MainTheoremA}
    Let $R$ be a perfectoid ring of characteristic $0$. Then the map
    \begin{equation*}
        J\mapsto J^{\flat} \defeq \set{f = (f^{(n)})_n}{f^{(n)} \in J ~\text{for all}~ n}
    \end{equation*}
    is an inclusion-preserving bijection between the set of ideals $J$ of $R$ such that the quotient $R/J$ is a perfectoid ring and the set of $p^{\flat}$-adically closed radical ideals of $R^{\flat}$. 
    
    Furthermore, the restriction of this bijection yields a homeomorphism between the subspace of $\Spec(R)$ consisting of prime ideals $\mathfrak{p}$ of $R$ such that the quotient $R/\mathfrak{p}$ is perfectoid and the subspace of $\Spec(R^{\flat})$ consisting of $p^{\flat}$-adically closed prime ideals of $R^{\flat}$.
\end{maintheorema}

As a consequence of the above theorem, we obtain a new and purely algebraic proof of the main theorem of \cite{dine2022Topologicala} (namely, \cite[Theorem 4.16]{dine2022Topologicala}).
This proof does not make use of the homeomorphism between the Berkovich spectrum of a perfectoid Tate ring and the Berkovich spectrum of its tilt.



\begin{notation} \label{Notation}

    For any element $t\in R$ in a ring $R$ and any $R$-module $M$ we denote by $M[t^{\infty}]$ the \textit{$t$-power-torsion submodule} of $M$; that is,
    \begin{equation*}
        M[t^{\infty}]\defeq \set{x \in M}{\exists n \in \setZ_{\geq 0}, t^n x = 0}.
    \end{equation*}
    Note that, for any $M$ and $t$, the quotient $M/M[t^{\infty}]$ is equal to the image of the canonical map $M\to M[1/t]$. For a ring $R$ of characteristic $0$, we denote the image $R/R[p^{\infty}]$ of the canonical map $R\to R[1/p]$ by $\overline{R}$.

    Let \(R\) be a ring and let \(a\) be an element of \(R\).
    An ideal \(I\) of \(R\) is \emph{\(a\)-adically closed} if it is a closed subset of the \(a\)-adically topological ring \(R\) with respect to the subspace topology.
    An ideal \(I\) of \(R\) is \emph{derived \(a\)-complete} if it is derived \(a\)-complete as an \(R\)-module.

    Let $A_{0}$ be a ring and $t\in A_{0}$ be any element.
    The Tate ring $A_{0}[1/t]$ is the Tate ring defined by the pair of definition $(A_{0}/A_{0}[t^{\infty}], (t))$.
    For a Tate ring $A=A_{0}[1/t]$, we denote by $\widehat{A^{u}}$ the \textit{uniform completion} of $A$ as defined in \cite[Definition 5.4]{ishizuka2023Mixed}. For any topological ring $A$ we denote by $A^{\circ}$ the set of power-bounded elements of $A$. 
    
    For a perfectoid ring \(R\) (in the sense of \cite[Definition 3.5]{bhatt2018Integral}),
    we call an element \(\pi \in R\) a \textit{perfectoid element} if \(R\) is \(\pi\)-adically complete and \(\pi^p\) divides \(p\) in \(R\).
    Such an element always exists in any perfectoid ring, by definition, but it is not unique.

    If $\pi\in R$ is a perfectoid element in a perfectoid ring $R$, then $R/R[\pi^{\infty}]$ is again a perfectoid ring (\cite[2.1.3]{cesnavicius2023Purity}) and the Tate ring $R[1/\pi]=(R/R[\pi^{\infty}])[1/\pi]$ is perfectoid in the sense of Fontaine \cite{fontaine2013Perfectoides} (recall that every complete Tate ring \(A\) can be endowed with a norm \(\norm{\cdot}_A\) defining its topology, and if $t\in A$ is a topologically nilpotent unit of $A$, one can choose the norm \(\norm{\cdot}_A\) such that $t$ is multiplicative with respect to \(\norm{\cdot}_A\), see, e.g., \cite[Definition 2.26 and Lemma 2.27]{nakazato2022Finite}; in particular, every complete Tate ring is a Banach ring in the sense of \cite{fontaine2013Perfectoides}). 

    Assume that $R$ is a perfectoid ring of characteristic $0$. Let us write $(W(R^{\flat}), (\xi))$ for the perfect prism corresponding to the perfectoid ring $R$ via \cite[Theorem 3.10]{bhatt2022Prismsa}, where $\xi$ is a distinguished element of $W(R^{\flat})$ with Witt vector coordinates $(\xi_{0}, \xi_{1}, \dots)$.
    By \cite[Lemma 3.9]{bhatt2018Integral}, there exists some unit multiple $vp \in R$ of $p$ such that $vp$ has a compatible system of $p$-power roots $\{(vp)^{1/p^{n}}\}_{n\geq 0}$ in $R$. Furthermore, one can adjust the choice of $\xi$ so that $p^{\flat}=\xi_{0}$ where \(p^\flat \defeq ((vp)^{1/p^n})_n \in R^\flat \cong \varprojlim_{x\mapsto x^{p}}R\), see \cite[2.1.2]{cesnavicius2023Purity}.   

    For every perfectoid ring $R$ of characteristic $0$, we fix an element $p^{\flat}\in R^{\flat}$ and a distinguished element $\xi\in W(R^{\flat})$ as above throughout this paper (in particular, we always assume that $\xi$ is chosen in such a way that $\xi_0=p^{\flat}$) and we consider $R$ (respectively, $R^{\flat}$) as topological rings equipped with the $p$-adic (respectively, the $p^{\flat}$-adic) topology.

\end{notation}

\subsection*{Authors' Contribution}
The second author conceived the idea of reconstructing the first author's paper \cite{dine2022Topologicala} in a purely algebraic approach and wrote a preliminary draft.
The first author refined the draft and clarified the connection with the previous result.

\subsection*{Acknowledgement}
We would like to thank Shinnosuke Ishiro, Kei Nakazato, and Kazuma Shimomoto for reading and providing feedback on this paper.


\section{Perfectoid Ideals}

In this section, we work with the following notion of perfectoid ideals.







\begin{definition}
    For any perfectoid ring \(R\),
    an ideal \(J\) of \(R\) is called a \textit{perfectoid ideal}
    if the quotient \(R/J\) is a perfectoid ring.
    Any perfectoid ideal is $p$-adically closed and derived \(p\)-complete since any perfectoid ring is \(p\)-adically complete.
\end{definition}



Recall from \cite{bhatt2022Prismsa} that a \textit{semiperfectoid ring} $S$ is a derived $p$-complete ring that can be written as a quotient of some perfectoid ring $R$.
Recall also from loc.~cit., Corollary 7.3, that, for every semiperfectoid ring $S$, there exists an initial object $S \to S_{\perfd}$ in the category of perfectoid $S$-algebras; the ring $S_{\perfd}$ is called the \textit{perfectoidization} of $S$.

\begin{lemma}
    Assume that \(R\) is a perfectoid ring and \(J\) is a derived \(p\)-complete ideal of \(R\).
    Then \(J\) is a perfectoid ideal if and only if \(J = J_{\perfd}\)
    where \(J_{\perfd} \defeq \ker(R \to (R/J)_{\perfd})\).
\end{lemma}

\begin{proof}
    Since \(J\) is derived \(p\)-complete, \(R/J\) is a semiperfectoid ring and thus
    the perfectoidization \((R/J)_{\perfd}\) is a universal perfectoid ring over \(R/J\).
    By \cite[Theorem 7.4]{bhatt2022Prismsa}, the canonical map \(R \to (R/J)_{\perfd}\) is surjective.
\end{proof}

To check the perfectoid property of rings and ideals,
we can sometimes use the following lemma.

\begin{lemma} \label{Semiperfectoid}
    Let $R$ be a semiperfectoid ring and suppose that there exists a perfectoid ring $R'$ and an injective map $R\hookrightarrow R'$. Then $R$ is a perfectoid ring.
    
    In particular, if a derived $p$-complete ideal $J$ in a perfectoid ring $R$ can be written as an intersection of perfectoid ideals, then $J$ is also a perfectoid ideal.
\end{lemma}

\begin{proof}
    We first deduce the second assertion of the lemma from the first. So, let $J$ be a derived $p$-complete ideal in a perfectoid ring $R$ which can be written as an intersection of a family $(J_{i})_{i}$ of perfectoid ideals.
    Any product of perfectoid rings is again perfectoid (\cite[Example 3.8(8)]{bhatt2019Regular}).
    In particular, $\prod_{i}R/J_{i}$ is a perfectoid ring, so the canonical map $R/J\hookrightarrow \prod_{i}R/J_{i}$ shows that the second assertion of the lemma is a consequence of the first. 

    Now let $R$ be a semiperfectoid ring admitting an injection $R\hookrightarrow R'$ into a perfectoid ring $R'$.
    By the universal property of perfectoidization, the injective map $R\hookrightarrow R'$ factors through the canonical map $R\to R_{\perfd}$.
    It follows that the canonical map $R\to R_{\perfd}$ is injective.
    But $R\to R_{\perfd}$ is also surjective by \cite[Theorem 7.4]{bhatt2022Prismsa}.
\end{proof}

In \Cref{CorrespSpecRed} we will clarify the relationship between the above notion of perfectoid ideals and the notion of spectrally reduced ideals in a seminormed ring as introduced in \cite{dine2022Topologicala}.
To do this, we define the notion of spectrally reduced ideals in the setting of Tate rings and explain how this definition relates to the previous one for seminormed rings.

\begin{definition}[cf. {\cite[Definition 2.16]{dine2022Topologicala}}] \label{DefSpecRed}
    For a Tate ring $A$, we call an ideal $\mcalJ \subsetneq A$ \textit{spectrally reduced} if 
    \begin{equation*}
        \mcalJ=\bigcap_{\ker(\phi)\supseteq \mcalJ}\ker(\phi)
    \end{equation*}
    where $\phi$ ranges over the continuous multiplicative seminorms on $A$ with $\ker(\phi)\supseteq \mcalJ$.
\end{definition}

\begin{remark} \label{RemarkSpecRed}
    By \cite[Remark 2.19]{dine2022Topologicala}, an ideal $\mcalJ$ of a Tate ring $A$ is spectrally reduced in the above sense if and only if $\mcalJ$ is a spectrally reduced ideal of the seminormed ring $(A, \norm{\cdot})$ in the sense of \cite[Definition 2.16]{dine2022Topologicala}, for some (or, equivalently, any) seminorm $\norm{\cdot}$ defining the topology on $A$ such that $A$ admits a topologically nilpotent unit which is multiplicative with respect to the seminorm $\norm{\cdot}$ (for the latter condition, see also \cite[Remark 2.13]{dine2022Topologicala}).

    For example, if $(A_{0}, (t))$ is a pair of definition of a Tate ring $A$, then, by definition, $t$ is a topologically nilpotent unit of $A$ which is multiplicative with respect to the canonical extension \(\norm{\cdot}_A\) to $A$ of the $t$-adic seminorm on $A_{0}$ (see \cite[Lemma 2.27]{nakazato2022Finite}).
\end{remark}

In the proof of \Cref{CorrespSpecRed} and in \Cref{SectionCorrespPrimeIdeals},
we use the \(t\)-saturation of ideals:

\begin{definition} \label{pSaturation}
    Let \(R\) be a ring, $t\in R$ and let $J$ be an ideal of \(R\).
    The \textit{\(t\)-saturation} \(\widetilde{J}^{t}\) of \(J\) is the ideal defined by
    \begin{equation*}
        \widetilde{J}^{t} \defeq \set{f \in R}{t^n f \in J \text{ for some \(n \geq 0\)}},
    \end{equation*}
    that is, \(\widetilde{J}^{t}\) is the inverse image of \(J[1/t]\) via the canonical map \(R \to R[1/t]\)
    and thus \(R/\widetilde{J}^{t}\) is \(t\)-torsion-free (more precisely, \(R/\widetilde{J}^{t}\)=\((R/J)/(R/J)[t^{\infty}]\)). If $t=p$, then we just write $\widetilde{J}$ instead of $\widetilde{J}^{p}$ for the $p$-saturation of $J$.
\end{definition}



\begin{lemma} \label{CorrespSpecRed}



    Let $R$ be a perfectoid ring of characteristic $0$, let $\pi\in R$ be a perfectoid element in $R$ and let $J$ be a derived $p$-complete ideal of $R$ such that $\pi\notin J$.
    If $J$ is a perfectoid ideal, then the ideal $J[1/\pi]$ of the perfectoid Tate ring $R[1/\pi]$ is spectrally reduced, and, if we assume that $R/J$ is $\pi$-torsion-free, the converse is also true.

    In general, $J$ is a perfectoid ideal if and only if $J$ is a radical ideal and the ideal $J[1/\pi]$ of the perfectoid Tate ring $R[1/\pi]$ is spectrally reduced.
\end{lemma}

\begin{proof}
    First suppose that $J$ is a perfectoid ideal of $R$. By definition, the quotient \(R/J\) is a perfectoid ring and thus the Tate ring \((R/J)[1/\pi]\) is a perfectoid Tate ring.
    Note that \((R/J)[1/\pi]\) is isomorphic (as a Tate ring) to the quotient of \(R[1/\pi]\) by the ideal \(J[1/\pi]\).
    Since every perfectoid Tate ring is uniform, we deduce that \(J[1/\pi]\) is a spectrally reduced ideal of $R[1/\pi]$.

    Conversely, first we assume that \(R/J\) is \(\pi\)-torsion-free and \(J[1/\pi]\) is spectrally reduced.
    Then the quotient \(R[1/\pi]/J[1/\pi] \cong (R/J)[1/\pi]\) is a perfectoid Tate ring, by a theorem of Bhatt and Scholze (see \cite[Theorem 4.4]{dine2022Topologicala}).
    By the proof of loc.~cit., $(R/J)_{\perfd}[1/\pi]$ is equal to the uniform completion of $(R/J)[1/\pi]$. Since perfectoid Tate rings are uniform, we conclude that the map $(R/J)[1/\pi]\to (R/J)_{\perfd}[1/\pi]$ is an isomorphism.
    The \(\pi\)-torsion-freeness of \(R/J\) implies that the canonical map \(R/J \to (R/J)[1/\pi] \cong (R/J)_{\perfd}[1/\pi]\) is injective
    and, in particular, \(R/J \to (R/J)_{\perfd}\) is injective.
    Since \(R/J \to (R/J)_{\perfd}\) is surjective by \cite[Theorem 7.4]{bhatt2022Prismsa}, \(R/J\) is isomorphic to \((R/J)_{\perfd}\) and thus \(J\) is a perfectoid ideal of \(R\).

    In the general case, suppose that $J$ is a derived $p$-complete radical ideal of $R$ such that the ideal $J[1/\pi]$ of $R[1/\pi]$ is spectrally reduced.
    Note that $\widetilde{J}^{\pi}[1/\pi]=J[1/\pi]$ where \(\widetilde{J}^{\pi}\) is the \(\pi\)-saturation of \(J\) as defined above in \Cref{pSaturation}.
    Since $J[1/\pi]$ is spectrally reduced, it is in particular (topologically) closed in \(R[1/\pi]\) and thus the isomorphism of Tate rings 
    \begin{equation*}
        (R/\widetilde{J}^{\pi})[1/\pi]=R[1/\pi]/\widetilde{J}^{\pi}[1/\pi]=R[1/\pi]/J[1/\pi]
    \end{equation*}
    shows that \((R/\widetilde{J}^{\pi})[1/\pi]\) is complete (in fact, we know that it is a perfectoid Tate ring, by \cite[Theorem 4.4]{dine2022Topologicala}).
    Consequently, since the complete Tate ring \((R/\widetilde{J}^{\pi})[1/\pi]\) has the pair of definition \((R/\widetilde{J}^{\pi}, (\pi))\), the quotient ring $R/\widetilde{J}^{\pi}$ is classically $\pi$-adically complete. Since, by the definition of a perfectoid element, we have $pR \subseteq \pi^{p}R\subseteq \pi R$, this entails that $R/\widetilde{J}^{\pi}$ is classically $p$-adically complete by \citeSta{090T}; a fortiori, $R/\widetilde{J}^{\pi}$ is derived $p$-complete. It follows that the ideal \(\widetilde{J}^{\pi}\) is also derived \(p\)-complete.
    Hence we can apply the assertion of the lemma in the $\pi$-torsion-free case to see that the ideal $\widetilde{J}^{\pi}$ is perfectoid.
    On the other hand, since $J$ is a radical ideal, it can be written as an intersection 
    \begin{equation*}
        J=\widetilde{J}^{\pi}\cap \bigcap_{\mathfrak{q}}\mathfrak{q},
    \end{equation*}
    where $\mathfrak{q}$ ranges over the prime ideals of $R$ which contain $J$ but do not contain $\widetilde{J}^{\pi}$.
    For such a prime ideal $\mathfrak{q}$, let $x\in \widetilde{J}^{\pi}$ be an element which does not belong to $\mathfrak{q}$. Then we can choose an integer $n>0$ such that $\pi^{n}x\in J\subseteq \mathfrak{q}$.
    Since $\mathfrak{q}$ is prime, we conclude that $\pi\in \mathfrak{q}$ and therefore $p\in \mathfrak{q}$.
    But this implies that every prime ideal $\mathfrak{q}$ as above is perfectoid by \Cref{IdealsEquivCharp} below.
    It follows that $J$ is perfectoid, by \Cref{Semiperfectoid} above.
\end{proof}



In the positive characteristic case,
perfectoid ideals are more simple objects.

\begin{proposition}[{cf. \cite[Remark 2.21]{dine2022Topologicala}}] \label{IdealsEquivCharp}
    Let \(R\) be a perfectoid ring and \(J\) be an ideal of \(R\).
    Assume that \(p \in J\).
    Then \(J\) is a perfectoid ideal if and only if it is a radical ideal.

    Assume that \(R\) is a perfect(oid) ring of positive characteristic \(p\)
    and $t\in R$ is an element such that $R$ is $t$-adically complete. Assume further that \(t \notin J\).
    Then the ideal \(J[1/t]\) of the complete Tate ring $R[1/t]$ is spectrally reduced if and only if \(J[1/t]\) is a closed radical ideal of \(R[1/t]\).
\end{proposition}

\begin{proof}
    Since \(R/J\) is of positive characteristic \(p\),
    the kernel of \(R/J \to (R/J)_{\perf} = (R/J)_{\perfd}\) is equal to \(\sqrt{J}/J\).
    The last statement follows from \cite[Remark 2.21]{dine2022Topologicala}.
\end{proof}

\section{Tilting and Untilting of Ideals}

In this section, we define the tilting and untilting of ideals,
which produces a correspondence between the set of perfectoid ideals of \(R\) and the set of $p^{\flat}$-adically closed radical ideals of \(R^\flat\) for any perfectoid ring \(R\) of characteristic $0$.
To do this, we start with some constructions related to perfectoid rings.

\begin{definition} \label{TiltRing}
    For any ring \(R\),
    we can define the \textit{tilt} of this ring,
    \begin{equation*}
        R^\flat \defeq \varprojlim_F R/pR,
    \end{equation*}
    where $F$ denotes the Frobenius on $R/pR$.
    If \(R\) is \(p\)-adically complete, the underlying multiplicative monoid of \(R^\flat\) is isomorphic to \(\varprojlim_{x \mapsto x^p} R\) as monoids,
    and then \(\varprojlim_{x \mapsto x^p} R\) has a natural ring structure induced from \(R^\flat\)
    by \cite[Lemma 3.2]{bhatt2018Integral}.
    The isomorphism is defined by
    \begin{align} \label{TiltIdentification}
        R^\flat = \varprojlim_F R/pR & \longrightarrow \varprojlim_{x \mapsto x^p} R \\
        x \defeq (\overline{x^{(n)}})_n & \longmapsto (\lim_{k \to \infty} (x^{(n+k)})^{p^k})_n.
    \end{align}
    We define \(x^\sharp \defeq \lim_{k \to \infty} (x^{(k)})^{p^k} \in R\) for any \(x = (\overline{x^{(n)}})_n \in R^\flat\).

    For any ideal \(J\) of \(R\),
    we have a \textit{theta map}
    \begin{align*}
        \theta_{R/J} \colon W((R/J)^\flat) & \longrightarrow R/J \\
        \sum_{n \geq 0} [x_n] p^n & \longmapsto \sum_{n \geq 0} x_n^\sharp p^n.
    \end{align*}
    If the Frobenius map on \((R/J)/p(R/J)\) is surjective,
    then \(\theta_{R/J}\) is surjective.
\end{definition}

Next, we define the tilting and untilting of ideals as follows.

\begin{definition} \label{TiltUntiltIdeal}
    Let $R$ be a classically $p$-adically complete ring.
    We identify the underlying multiplicative monoid of the tilt $R^{\flat}$ of $R$ with $\varprojlim_{x\mapsto x^{p}}R$ via (\ref{TiltIdentification}), so an element $x\in R^{\flat}$ is a sequence $(x^{(n)})_{n}$ of elements of $R$ with 
        \begin{equation*}
            (x^{(n+1)})^p=x^{(n)} \in R
        \end{equation*}
        for all $n\in\setZ_{\geq 0}$.
        Note that under this identification we have $x^{\sharp}=x^{(0)}$ for any $x\in R^{\flat}$.
        For any subset $J$ of $R$, we define a subset $J^{\flat}$ of $R^{\flat}$ by 
        \begin{equation} \label{DefTilt}
            J^\flat \defeq \set{x = (x^{(n)})_n \in R^\flat}{x^{(n)} \in J~\text{for any}~n \geq 0}.
        \end{equation}

        If $J$ is a $p$-adically closed ideal of $R$, then $J^{\flat}$ is an ideal of $R^{\flat}$.
        If moreover $J$ is a \(p\)-adically closed radical ideal, we have $J^{\flat}=\set{x \in R^\flat}{x^\sharp \in J}$.
        We call $J^{\flat}$ the \textit{tilt} of the subset (respectively, of the $p$-adically closed ideal) $J$ of $R$.
        Remark that this operation \((-)^\flat\) maps prime ideals of \(R\) to prime ideals of \(R^\flat\) by simple observations.


    Conversely, if \(R\) is a perfectoid ring and \(I\) is an ideal of \(R^\flat\),
    we define a subset of \(W(R^\flat)\)
    \begin{equation} \label{DefWittIdeal}
        W(I) \defeq \set{\sum_{n \geq 0} [x_n] p^n \in W(R^\flat)}{x_n \in I}
    \end{equation}
    which is well-defined by the unique representation.
    In particular, if $I$ is a radical ideal of $R^{\flat}$ (so that the quotient $R^{\flat}/I$ is perfect), this \(W(I)\) is equal to \(\ker(W(R^\flat) \to W(R^\flat/I))\)
    and thus \(W(I)\) is an ideal of \(W(R^\flat)\).
    Then we can define an ideal of \(R\)
    \begin{equation} \label{DefUntilt}
        I^\sharp \defeq \theta_R(W(I))
    \end{equation}
    because of the surjectivity of \(\theta_R\).
    We call it the \textit{untilt} of \(I\).

    In the next section (\Cref{SharpPrimetoPrime}), we will show that $(-)^{\sharp}$ maps $p^{\flat}$-adically closed prime ideals of $R^{\flat}$ to prime ideals of $R$ for any perfectoid ring $R$ of characteristic $0$, but we \textit{do not} use this fact in this section.
\end{definition}

One of the important properties of these operations is the following.
Roughly speaking, the tilting and untilting operations preserve the perfectoid property of ideals under some topologically closed assumptions.

\begin{lemma} \label{TiltandUntiltCorresp}
    For any $p$-adically complete ring $R$ and any \(p\)-adically closed ideal $J$ of $R$, the tilt $J^{\flat}$ of $J$ is equal to the kernel of the homomorphism 
    \begin{equation*}
        \map{\varphi^\flat}{R^\flat}{(R/J)^\flat}
    \end{equation*}
    corresponding to the canonical quotient map $R\to R/J$.

    If $R$ is a perfectoid ring of characteristic $0$ and $J$ is a perfectoid ideal of $R$, then $\varphi^{\flat}$ induces an isomorphism $R^{\flat}/J^{\flat} \xrightarrow{\cong} (R/J)^{\flat}$.
    In particular, the tilt $J^{\flat}$ of any perfectoid ideal $J$ of a perfectoid ring $R$ is a $p^{\flat}$-adically closed radical ideal of $R^{\flat}$. 
    
    Conversely, for any perfectoid ring $R$ of characteristic \(0\) and any $p^{\flat}$-adically closed radical (and thus perfectoid) ideal $I$ of $R^{\flat}$, the untilt $I^{\sharp}$ is also a perfectoid ideal of $R$.


\end{lemma}

\begin{proof}
    The equality $J^{\flat}=\ker(\varphi^{\flat})$ follows by the same argument as in the proof of \cite[Lemma 4.8]{dine2022Topologicala}.
    If $R$ and $R/J$ are perfectoid, their tilts are $p^{\flat}$-adically complete (see, for example, \cite[2.1.2]{cesnavicius2023Purity}).
    Then the kernel \(J^\flat\) is \(p^\flat\)-adically closed.
    By topological Nakayama's lemma,
    \(R^\flat \to (R/J)^\flat\) is surjective and this proves that \(\varphi^\flat\) induces an isomorphism \(R^\flat/J^\flat \cong (R/J)^\flat\) and \(J^\flat\) is a $p^{\flat}$-adically closed radical ideal of \(R^\flat\).

    For a \(p^\flat\)-adically closed radical ideal \(I \subseteq R^\flat\),
    the canonical map \(W(R^\flat) \to W(R^\flat/I)\) is surjective by topological Nakayama's lemma
    and the kernel is given by \(W(I)\) by definition.
    Then we have
    \begin{equation*}
        R/I^\sharp \cong W(R^\flat)/((\xi) + W(I)) \cong W(R^\flat/I)/(\xi).
    \end{equation*}
    Since \(I\) is a \(p^\flat\)-adically closed radical ideal of \(R^\flat\),
    the quotient \(R^\flat/I\) is a \(p^\flat\)-adically complete\footnote{This \(p^\flat\)-adic completeness is necessary because we want the pair \((W(R^\flat/I), (\xi))\) to be a (perfect) prism, see also \Cref{UntiltandClosedness}.} perfect(oid) \(R^\flat\)-algebra.
    Then \(R/I^\sharp\) is a perfectoid \(R\)-algebra by \cite[Proposition 2.1.9]{cesnavicius2023Purity} or \cite[Th\`eor\'eme 2.10]{dospinescu2023Conjecture}.
\end{proof}

\begin{remark}\label{UntiltandClosedness}
    The assumption that the ideal $I$ in \Cref{TiltandUntiltCorresp} be $p^{\flat}$-adically closed is not automatic: Indeed, we show in \Cref{CounterexamplePerfdIdeal} that there exists a perfectoid ring $R$ of characteristic $0$ such that not all prime ideals of $R^{\flat}$ are $p^{\flat}$-adically closed. Furthermore, the assumption that $I$ be $p^{\flat}$-adically closed is also necessary for the assertion of \Cref{TiltandUntiltCorresp} to hold true: If $I$ is a radical ideal of $R^{\flat}$ which is not $p^{\flat}$-adically closed in $R^{\flat}$, then $I^{\sharp}$ is not a perfectoid ideal of $R$. Indeed, if $I^{\sharp}$ is a perfectoid ideal, i.e., $R/I^{\sharp}\cong W(R^{\flat}/I)/(\xi)$ is a perfectoid ring, then, by \cite[Theorem 3.10]{bhatt2022Prismsa}, the pair $(W(R^{\flat}/I), (\xi))$ is a perfect prism. By loc.~cit., Lemma 3.8(2), this entails that $W(R^\flat/I)$ is classically $(p, \xi)$-adically complete. But then $R^{\flat}/I=W(R^{\flat}/I)/(p)$ is classically $\xi_0=p^{\flat}$-adically complete and thus $I$ must be $p^{\flat}$-adically closed.
\end{remark}

Using the above observations, we can show the following one-to-one correspondence between the sets of perfectoid ideals and \(p^\flat\)-adically closed radical ideals.

\begin{theorem} \label{CorrespPerfdIdeals}
    Let $R$ be a perfectoid ring of characteristic $0$.
    Then the tilting and untilting operations $(-)^{\flat}$ and $(-)^\sharp$ define a one-to-one correspondence between the set of perfectoid ideals of $R$ and the set of $p^{\flat}$-adically closed radical ideals of $R^{\flat}$.
\end{theorem}

\begin{proof}
    By \Cref{TiltandUntiltCorresp}, it suffices to show that the maps \((-)^\flat\) and \((-)^\sharp\) are inverse to each other.
    Take any perfectoid ideal \(J\) of \(R\).
    The perfectoid \(R\)-algebra \(R/J\) is the form \(W((R/J)^\flat)/\xi W((R/J)^\flat)\) by \cite[Proposition 2.1.9]{cesnavicius2023Purity} or \cite[Th\'eor\`eme 2.10]{dospinescu2023Conjecture}.
    Then the commutative diagram
    \begin{center}
        \begin{tikzcd}
            W(R^\flat)/W(J^\flat) \arrow[r, "\cong"]  & W((R/J)^\flat) \arrow[dd, "\theta_{R/J}", two heads] \\
            W(R^\flat) \arrow[d, "\theta_R", two heads] \arrow[u, two heads] &  \\
            R \arrow[r, two heads] & R/J
        \end{tikzcd}
    \end{center}
    and the fact that \(\ker(\theta_R)\) and \(\ker(\theta_{R/J})\) are generated by the same \(\xi\)
    show that \(J^{\flat \sharp} = J\).

    Conversely, if $I$ is a $p^{\flat}$-adically closed radical ideal of $R^{\flat}$, we have
    \begin{equation*}
        R^\flat/I^{\sharp \flat} \cong (R/I^\sharp)^\flat \cong (W(R^\flat/I)/(\xi))^\flat \cong R^\flat/I
    \end{equation*}
    canonically and thus \(I=I^{\sharp \flat}\).
\end{proof}
We deduce from the above theorem a corollary which ensures that the theorem also stays true when the role of $p$ is taken on by another perfectoid element of $R$. Recall that for any perfectoid element $\pi$ in a perfectoid ring $R$ of characteristic $0$, \cite[Lemma 3.9]{bhatt2018Integral} entails that there exists a unit $u\in R^{\times}$ such that $u\pi$ admits a compatible system of $p$-power roots in $R$, i.e., there exists an element $\pi^{\flat}\in R^{\flat}$ such that $\pi^{\flat\sharp}$ is a unit multiple of $\pi$ in $R$.

\begin{corollary}\label{Perfectoid elements and untilting}
    Let $R$ be a perfectoid ring of characteristic $0$, let $\pi\in R$ be a perfectoid element and let $\pi^{\flat}\in R^{\flat}$ be an element such that $\pi^{\flat\sharp}$ is a unit multiple of $\pi$ in $R$. Then a radical ideal $I$ of $R^{\flat}$ is $p^{\flat}=\xi_0$-adically closed if and only if it is $\pi^{\flat}$-adically closed.
\end{corollary}

\begin{proof}
    If $I$ is a $p^{\flat}$-adically closed radical ideal of $R^{\flat}$, then, by \Cref{CorrespPerfdIdeals}, $I=J^{\flat}$ for some perfectoid ideal $J$ of $R$. By \Cref{TiltandUntiltCorresp}, $R^{\flat}/I=(R/J)^{\flat}$. By \cite[2.1.2]{cesnavicius2023Purity}, $(R/J)^{\flat}$ is $\pi^{\flat}$-adically complete. Consequently, $I$ is $\pi^{\flat}$-adically closed.

Conversely, once again by the discussion in \cite[2.1.2]{cesnavicius2023Purity}, $(\pi^{\flat})^{p}$ divides $p^{\flat}$ in $R^{\flat}$ and thus $p^{\flat} R^{\flat} \subseteq (\pi^{\flat})^p R^{\flat} \subseteq \pi^{\flat} R^{\flat}$. This entails that every $\pi^{\flat}$-adically complete $R^{\flat}$-algebra is also $p^{\flat}$-adically complete, by \citeSta{090T}. Hence every $\pi^{\flat}$-adically closed ideal $I$ of $R^{\flat}$ is also $p^{\flat}$-adically closed, as claimed.
\end{proof}

\section{Correspondence of Prime Ideals} \label{SectionCorrespPrimeIdeals}

For any perfectoid ring $R$, the tilt of a $p$-adically closed prime ideal of $R$ is a prime ideal of $R^{\flat}$. This can be shown easily. However, the stability of (closed) prime ideals under the untilting operation $(-)^\sharp$ is not quite obvious.
We show this fact by using an idea that appeared in the proof of \cite[Theorem 4.16]{dine2022Topologicala}.



\begin{definition}[cf. {\cite[Definition 3.13]{dine2022Topologicala}}]
    Let $A$ be a Tate ring and let $\mcalJ$ be an ideal of $A$ which is not dense in $A$.
    The \textit{spectral radical} $\mcalJ_{\sprad}$ of $\mcalJ$ is the smallest spectrally reduced ideal of $A$ containing $\mcalJ$.
    That is, 
    \begin{equation*}
        \mcalJ_{\sprad}=\bigcap_{\ker(\phi)\supseteq \mcalJ}\ker(\phi),
    \end{equation*}
    where $\phi$ ranges over the continuous multiplicative seminorms $\map{\phi}{A}{\setR_{\geq 0}}$ with the property that $\ker(\phi)\supseteq \mcalJ$.
\end{definition}

\begin{remark}
    Since any Tate ring can be equipped with a seminorm defining its topology (see \Cref{Notation} or \cite[Lemma 2.27]{nakazato2022Finite}), every non-dense ideal is contained in some ideal of the form $\ker(\phi)$ for a continuous multiplicative seminorm:
    This is a consequence of the fact that the Berkovich spectrum of any seminormed ring is non-empty (see \cite[Lemma 3.12]{dine2022Topologicala}).
    This shows that the above notion of the spectral radical is well-defined.
    Furthermore, the spectral radical $\mcalJ_{\sprad}$ in the sense of the above definition agrees with the spectral radical of the seminormed ring $(A, \norm{\cdot})$ if $\norm{\cdot}$ is some (or, equivalently, any) seminorm defining the topology of the Tate ring $A$ and admitting a multiplicative topologically nilpotent unit $t$ (see \cite[Remark 2.19]{dine2022Topologicala} or \Cref{RemarkSpecRed}).
    For example, we could let $\norm{\cdot}$ be the seminorm defined by some pair of definition $(A_{0}, (t))$ of $A$.
\end{remark} 

The following lemma shows the relationship between \((-)_{\perfd}\) and \((-)_{\sprad}\).

\begin{lemma} \label{PerfdIsSpecrad}
    Let \(R\) be a perfectoid ring of characteristic \(0\) and \(J\) be a derived \(p\)-complete ideal of \(R\).
    Then the \(p\)-saturation \(\widetilde{J_{\perfd}}\) of \(J_{\perfd}\) in \(R\) is equal to the inverse image of the spectral radical \((J[1/p])_{\sprad}\) of \(J[1/p]\) in \(\overline{R}[1/p]\) via the canonical map \(R \to \overline{R}[1/p]\).
\end{lemma}

\begin{proof}
    Since \(J\) is derived \(p\)-complete, we can apply \cite[Theorem 5.2]{ishizuka2023Calculation} for \(R/J\).
    Then we have
    \begin{equation*}
        J_{\perfd}[1/p] = \ker(\overline{R}[1/p] \to \overline{(R/J)_{\perfd}}[1/p]) = \ker(\overline{R}[1/p] \to \widehat{((R/J)[1/p])^u})
    \end{equation*}
    and the last term is equal to \((J[1/p])_{\sprad}\)
    because of the definition of the spectral radical and of the uniform completion \cite[Remark 5.7]{ishizuka2023Mixed}.
\end{proof}

\begin{lemma} \label{IntersectionpSatPerfdIdeal}
    Let \(I\) and \(J\) be derived \(p\)-complete ideals of a perfectoid ring \(R\) of characteristic \(0\).
    Then we have
    \begin{equation*}
        \widetilde{(I \cap J)_{\perfd}} = \widetilde{I_{\perfd}} \cap \widetilde{J_{\perfd}}.
    \end{equation*}
\end{lemma}

\begin{proof}
    The intersection \(I \cap J\) is derived \(p\)-complete since it is the kernel of the map of derived \(p\)-complete rings \(R \to (R/I) \times (R/J)\).
    Therefore, the assertion follows from \Cref{PerfdIsSpecrad} and the fact that the map \((-)_{\sprad}\) on ideals of a perfectoid Tate ring commutes with intersections.
\end{proof}

\begin{lemma} \label{pSatStabilityPerfdIdeal}
    Let \(J\) be a perfectoid ideal of a perfectoid ring \(R\).
    Then the \(p\)-saturation \(\widetilde{J}\) of \(J\) in \(R\) is also a perfectoid ideal.
\end{lemma}

\begin{proof}
    Note that an element \(f\) of \(R\) is in \(\widetilde{J}\) if and only if the image of \(f\) in \(R/J\) is \(p\)-power torsion.
    That is, \(R/\widetilde{J}\) is the quotient of \(R/J\) by its \(p\)-power torsion part \((R/J)[p^\infty]\).
    By \cite[2.1.3]{cesnavicius2023Purity} or \cite[Proposition 2.19]{dospinescu2023Conjecture}, we see that \(\widetilde{J}\) is a perfectoid ideal.
\end{proof}

\begin{lemma} \label{TopClsIntersection}
    Let \(I\) and \(J\) be ideals of a \(t\)-adically separated ring \(R\) for some element \(t \in R\).
    Let us denote by \(\overline{\mfraka}^t\) the (topological) closure of an ideal \(\mfraka\) of \(R\) with respect to the \(t\)-adic topology of \(R\).
    Then we have
    \begin{equation*}
        \overline{I}^t \cap \overline{J}^t \subseteq \sqrt{\overline{I \cap J}^t}.
    \end{equation*}
\end{lemma}

\begin{proof}
    Let \(f\) be an element of \(\overline{I}^t \cap \overline{J}^t\).
    Choose sequences \((f_n) \subset I\) and \((g_n) \subset J\) such that \(f = \lim_n f_n = \lim_n g_n\) in the \(t\)-adically separated ring \(R\).
    Then \((f_n g_n)\) is a sequence in \(IJ \subseteq I \cap J\) and \(f^2 = (\lim_n f_n) (\lim_n g_n) = \lim_n f_n g_n\)
    is in \(\overline{I \cap J}^t\).
\end{proof}

\begin{proposition} \label{SharpPrimetoPrime}
    Let \(R\) be a perfectoid ring of characteristic \(0\) and let \(\mfrakq\) be a \(p^\flat\)-adically closed prime ideal of \(R^\flat\).
    Then the untilt \(\mfrakq^\sharp\) is a prime ideal of \(R\).
\end{proposition}

\begin{proof}
    In view of the canonical isomorphism \(R/p R \cong R^\flat/p^\flat R^\flat\),
    the assertion is clear if \(p^\flat \in \mfrakq\) (or equivalently, if \(p \in \mfrakq^\sharp\)).
    Hence we may assume that \(p^\flat \notin \mfrakq\).
    We proceed as in the proof of \cite[Theorem 4.16]{dine2022Topologicala}.
    So, let \(I\) and \(J\) be ideals of \(R\) with \(IJ \subseteq \mfrakq^\sharp\) and it suffices to show that either \(I\) or \(J\) is contained in \(\mfrakq^\sharp\).

    By \Cref{TiltandUntiltCorresp},
    the untilt \(\mfrakq^\sharp\) is a perfectoid ideal of \(R\).
    Then \(\mfrakq^\sharp\) is a radical ideal by the reduced property of perfectoid rings (see \cite[2.1.3]{cesnavicius2023Purity}) and
    the inclusion \(IJ \subseteq \mfrakq^\sharp\) implies \(I \cap J \subseteq \mfrakq^\sharp\).
    Since any perfectoid ideal of \(R\) is \(p\)-adically closed in \(R\),
    we have
    \begin{equation*}
        \sqrt{\overline{I \cap J}^p} \subseteq \mfrakq^\sharp
    \end{equation*}
    where \(\overline{I \cap J}^p\) is the (topological) closure of \(I \cap J\) with respect to the \(p\)-adic topology of \(R\).
    By \Cref{TopClsIntersection},
    we have
    \begin{equation*}
        \overline{I}^p \cap \overline{J}^p \subseteq \mfrakq^\sharp.
    \end{equation*}
    By \Cref{IntersectionpSatPerfdIdeal},
    we see that
    \begin{equation*}
        \widetilde{(\overline{I}^p)_{\perfd}} \cap \widetilde{(\overline{J}^p)_{\perfd}} = \widetilde{(\overline{I}^p \cap \overline{J}^p)_{\perfd}}.
    \end{equation*}
    Since the map \((-)^\flat\) from subsets of \(R\) to subsets of \(R^\flat\) preserves inclusions and intersections,
    this means that
    \begin{equation*}
        (\widetilde{(\overline{I}^p)_{\perfd}})^\flat \cap (\widetilde{(\overline{J}^p)_{\perfd}})^\flat \subseteq (\widetilde{\mfrakq^\sharp})^\flat.
    \end{equation*}
    By the definition of the tilting operation (\ref{DefTilt}),
    if \(f \in (\widetilde{\mfrakq^\sharp})^\flat\),
    then \(p^n f^\sharp \in \mfrakq^\sharp\) for some \(n \geq 0\).
    Using the fact that \(p^n = ((p^\flat)^n)^\sharp \in R\) and \Cref{CorrespPerfdIdeals},
    we have \((p^\flat)^n f \in \mfrakq^{\sharp \flat} = \mfrakq\) and thus \(f \in \mfrakq\) because of the assumption that \(p^\flat \notin \mfrakq\).
    This shows that
    \begin{equation*}
        (\widetilde{(\overline{I}^p)_{\perfd}})^\flat \cap (\widetilde{(\overline{J}^p)_{\perfd}})^\flat \subseteq \mfrakq.
    \end{equation*}
    Since \(\mfrakq\) is a prime ideal of \(R^\flat\),
    we see that either
    \begin{equation*}
        (\widetilde{(\overline{I}^p)_{\perfd}})^\flat \subseteq \mfrakq \text{ or } (\widetilde{(\overline{J}^p)_{\perfd}})^\flat \subseteq \mfrakq.
    \end{equation*}
    On the other hand,
    \Cref{pSatStabilityPerfdIdeal} ensures that \(\widetilde{(\overline{I}^p)_{\perfd}}\) and \(\widetilde{(\overline{J}^p)_{\perfd}}\) are perfectoid ideals of \(R\).
    Applying the inclusion-preserving map \((-)^\sharp\) and using \Cref{CorrespPerfdIdeals},
    we conclude that
    \begin{equation*}
        I \subseteq \widetilde{(\overline{I}^p)_{\perfd}} \subseteq \mfrakq^\sharp \text{ or } J \subseteq \widetilde{(\overline{J}^p)_{\perfd}} \subseteq \mfrakq^\sharp,
    \end{equation*}
    as desired.
\end{proof}

\section{The Perfectoid Spectrum}

We will show that the one-to-one correspondence in \Cref{CorrespPerfdIdeals} induces a homeomorphism between a certain subspace of the spectrum of $R$ and a certain subspace of the spectrum of $R^{\flat}$.

\begin{definition} \label{PerfdSpec}
    For a perfectoid ring $R$ of characteristic $0$, the perfectoid spectrum of $R$ is the subspace
    \begin{equation*}
        \Spec_{\perfd}(R) \defeq\{\, \mathfrak{p}\in\Spec(R)\mid \mathfrak{p}~\textrm{is a perfectoid ideal}\,\},
    \end{equation*}
    of $\Spec(R)$ endowed with the subspace topology induced from the Zariski topology on $\Spec(R)$. The perfectoid spectrum of $R^{\flat}$ is the subspace
        \begin{equation*}
            \Spec_{\perfd}(R^{\flat}) \defeq\{\, \mathfrak{p}\in\Spec(R^{\flat})\mid \mathfrak{p}~\textrm{is}~p^{\flat}\textrm{-adically closed}\,\}
        \end{equation*}
    of $\Spec(R^{\flat})$, again with the subspace topology induced from the Zariski topology.

\end{definition}

We recall a variant of the notion of the topological spectrum from \cite{dine2022Topologicala}.

\begin{definition}[cf. {\cite[Definition 2.18]{dine2022Topologicala}}]
    The \textit{topological spectrum} of a Tate ring $A$ is the subspace 
    \begin{equation*}
        \Spec_{\Top}(A) \defeq \set{\mfrakp \in \Spec(A)}{\text{\(\mfrakp\) is spectrally reduced}}
    \end{equation*}
    of $\Spec(A)$, where the topology on $\Spec_{\Top}(A)$ is the subspace topology induced from the Zariski topology on $\Spec(A)$.
\end{definition}
We note that, by \cite[Remark 2.19]{dine2022Topologicala}, for any seminorm $\norm{\cdot}$ defining the topology on $A$ such that there exists a topologically nilpotent unit $t\in A$ multiplicative with respect to $\norm{\cdot}$, the topological spectrum of $A$ defined as above is equal to the topological spectrum of the seminormed ring $(A, \norm{\cdot})$ in the sense of \cite[Definition 2.18]{dine2022Topologicala}. 

First, we record the following observations about the perfectoid spectrum and the topological spectrum.




\begin{lemma} \label{EqualitySomeSpectrum}
    Let $R$ be a perfectoid ring of characteristic $0$, let $\pi$ be a perfectoid element of $R$ and let $\pi^{\flat}$ be an element of $R^{\flat}$ such that $\pi^{\flat\sharp}$ is a unit multiple of $\pi$ in $R$ (by \cite[Lemma 3.9]{bhatt2018Integral}, such an element $\pi^\flat$ always exists). Then $\Spec_{\Top}(R[1/\pi])$ is equal to the open subspace $D(\pi)\cap \Spec_{\perfd}(R)$ of $\Spec_{\perfd}(R)$ and $\Spec_{\Top}(R^{\flat}[1/\pi^{\flat}])$ is equal to the open subspace $D(\pi^{\flat})\cap \Spec_{\perfd}(R^{\flat})$ of $\Spec_{\perfd}(R^{\flat})$.
\end{lemma}

\begin{proof}
    By \Cref{Perfectoid elements and untilting} a prime ideal of $R^{\flat}$ is $\pi^{\flat}$-adically closed if and only if it is $p^{\flat}$-adically closed. Hence the assertion follows from \Cref{CorrespSpecRed} and \Cref{IdealsEquivCharp}.
\end{proof}

Combining \Cref{CorrespPerfdIdeals} and \Cref{SharpPrimetoPrime}, we obtain a correspondence (\Cref{CorrespPrimePerfdIdeals} below) between perfectoid prime ideals of a perfectoid ring \(R\) of characteristic $0$ and \(p^\flat\)-adically closed prime ideals of \(R^\flat\).
This correspondence is indeed a homeomorphism, not only a bijection. To prove this, we record the following lemma.

\begin{lemma} \label{ClosuresOfRadicalIdeals}
    Let $R$ be a perfect(oid) ring of characteristic $p$ and let $I, J$ be ideals of $R$.
    If $J$ is a radical ideal, then so is $\overline{J}^{I}$, the closure of $J$ with respect to the $I$-adic topology on $R$. 
\end{lemma}
\begin{proof}
    The quotient $R/\overline{J}^{I}$ is the $I$-adic completion of the perfect ring $R/J$, so the lemma follows from \cite[Lemma 3.6]{dietz2007Big}.
\end{proof}

\begin{theorem} \label{CorrespPrimePerfdIdeals}
    Let $R$ be a perfectoid ring of characteristic $0$.
    Then the tilting map (\ref{DefTilt})
    \begin{equation*}
        \mathfrak{p} \mapsto \mathfrak{p}^{\flat} = \set{f = (f^{(n)})_n \in R^\flat}{f^\sharp \in \mfrakp}
    \end{equation*}
    induces a homeomorphism 
    \begin{equation*}
        \Spec_{\perfd}(R) \xrightarrow{\simeq} \Spec_{\perfd}(R^{\flat}),
    \end{equation*}
    whose inverse is given by the untilting map $\mathfrak{q}\mapsto\mathfrak{q}^\sharp$ defined in (\ref{DefUntilt}).
\end{theorem}

\begin{proof}
    It is readily seen that the map $J\mapsto J^{\flat}$ takes ($p$-adically closed) prime ideals of $R$ to prime ideals of $R^{\flat}$.
    By \Cref{SharpPrimetoPrime}, the map $I\mapsto I^\sharp$ also takes ($p^{\flat}$-adically closed) prime ideals of $R^{\flat}$ to (perfectoid) prime ideals of $R$. 
    By \Cref{CorrespPerfdIdeals}, the two maps $J \mapsto J^{\flat}$ and $I\mapsto I^\sharp$ are inverse to each other, so we obtain a bijection $\Spec_{\perfd}(R) \xrightarrow{\simeq} \Spec_{\perfd}(R^{\flat})$. 

    To see that this bijection is a homeomorphism, note that a perfectoid prime ideal $\mathfrak{p}$ of $R$ contains an ideal $J$ if and only if it contains the perfectoid ideal $(\overline{J}^{p})_{\perfd}$:
    Indeed, if $J \subseteq \mathfrak{p}$, we have $\overline{J}^{p}\subseteq \mathfrak{p}$ since perfectoid ideals are $p$-adically closed, and then the map of semiperfectoid rings 
        \begin{equation*}
            R/\overline{J}^{p} \to R/\mathfrak{p}
        \end{equation*}
    factors through $R/(\overline{J}^{p})_{\perfd}$ by the universal property of perfectoidization. 
    Hence every closed subset of $\Spec_{\perfd}(R)$ is of the form 
    \begin{equation*}
        \mathcal{V}_{\Spec_{\perfd}(R)}(J)= \set{\mfrakp \in \Spec_{\perfd}(R)}{\mfrakp \supseteq J}
    \end{equation*}
    for some perfectoid ideal $J$ of $R$.

    Using \Cref{ClosuresOfRadicalIdeals}, every closed subset of $\Spec_{\perfd}(R^{\flat})$ is of the form $\mathcal{V}_{\Spec_{\perfd}(R^{\flat})}(I)$ for some $p^{\flat}$-adically closed radical ideal \(I\) of $R^{\flat}$.
    Since the mutually inverse maps $J \mapsto J^{\flat}$ and $I \mapsto I^\sharp$ are inclusion-preserving, we see that 
    \begin{equation*}
        (\mathcal{V}_{\Spec_{\perfd}(R)}(J))^{\flat} = \mathcal{V}_{\Spec_{\perfd}(R^{\flat})}(J^{\flat})
    \end{equation*}
    and 
    \begin{equation*}
        (\mathcal{V}_{\Spec_{\perfd}(R^{\flat})}(I))^\sharp = \mathcal{V}_{\Spec_{\perfd}(R)}(I^\sharp)
    \end{equation*}
    for any perfectoid ideal $J$ of $R$ and any $p^{\flat}$-adically closed radical ideal $I$ of $R^{\flat}$.
    It follows that the map 
    \begin{equation*}
        \Spec_{\perfd}(R) \to \Spec_{\perfd}(R^{\flat}), \mathfrak{p} \mapsto \mathfrak{p}^{\flat},
    \end{equation*}
    is a homeomorphism and its inverse map is \(\mfrakq \mapsto \mfrakq^\sharp\).
\end{proof}

Let us discuss the relationship between the above theorem and the main result of \cite{dine2022Topologicala}. 
Recall that the Berkovich spectrum $\mathcal{M}(A)$ of a seminormed ring $(A, \norm{\cdot})$ (introduced for Banach rings in \cite{berkovich2012Spectral}) is the topological space of bounded multiplicative seminorms $\map{\phi}{A}{\setR_{\geq 0}}$ on $A$, endowed with the weakest topology making each of the maps 
\begin{equation*}
    \mathcal{M}(A) \to \mathbb{R}_{\geq 0}, \phi \mapsto \phi(f),
\end{equation*}
for $f\in A$, continuous.

We note that if there is a topologically nilpotent unit $t\in A$ which is multiplicative with respect to the seminorm on $A$, then, by \cite[Lemma 2.11]{dine2022Topologicala}, for every continuous multiplicative seminorm $\map{\phi}{A}{\setR_{\geq 0}}$, there exists $s \in (0,1)$ such that $\phi^{s} \in \mathcal{M}(A)$ (in fact, if the seminorm $\phi$ satisfies $\phi(t) = \norm{t}$, then $\phi\in \mathcal{M}(A)$, see the proof of loc.~cit.).

If $A$ is a perfectoid Tate ring, viewed as a uniform Banach ring by choosing a power-multiplicative norm defining the topology on $A$, then we have a canonical map 
\begin{equation*}
    \mathcal{M}(A) \to \mathcal{M}(A^{\flat}), \phi \mapsto \phi^{\flat},
\end{equation*}
where the seminorm $\map{\phi^{\flat}}{A^\flat}{\setR_{\geq 0}}$ is defined by $\phi^{\flat}(f) \defeq \phi(f^\sharp)$ for all $f\in A^{\flat}$. 
By \cite[Theorem 3.3.7(c)]{kedlaya2015Relativea} (which is an analogue for Berkovich spectra of \cite[Corollary 6.7(iii)]{scholze2012Perfectoida}), this map is a homeomorphism $\mathcal{M}(A) \xrightarrow{\simeq} \mathcal{M}(A^{\flat})$, for every perfectoid Tate ring $A$;
we denote by $\phi \mapsto \phi^\sharp$ its inverse. 

The following proposition shows that the untilting operation $(-)^\sharp$ introduced in this paper (\ref{DefUntilt}) generalizes the analytically defined untilting operation (denoted by the same symbol) which was introduced in the context of perfectoid Tate rings in \cite[Section 4]{dine2022Topologicala}.

\begin{proposition} \label{TwoUntiltingMaps}
    Let $R$ be a perfectoid ring of characteristic $0$, let $\pi$ be a perfectoid element of $R$ and choose an element $\pi^{\flat}$ of $R^{\flat}$ such that $\pi^{\flat\sharp}$ is a unit multiple of $\pi$ in $R$ (this is always possible by \cite[Lemma 3.9]{bhatt2018Integral}). Let $I$ be a $p^{\flat}$-adically closed (or, equivalently, $\pi^{\flat}$-adically closed, see \Cref{Perfectoid elements and untilting}) radical ideal of $R^{\flat}$ such that $R^{\flat}/I$ is $\pi^{\flat}$-torsion-free. 
    Endow the perfectoid Tate ring $R[1/\pi]$ (respectively, $(R[1/\pi])^{\flat}=R^{\flat}[1/\pi^{\flat}]$ \footnote{This isomorphism is deduced from \cite[2.1.3]{cesnavicius2023Purity}}) with a power-multiplicative norm $\norm{\cdot}$ defining its topology such that the element $\pi \in R[1/\pi]$ (respectively, $\pi^{\flat} \in (R[1/\pi])^{\flat}$) is a multiplicative element with respect to $\norm{\cdot}$ (note that such a power-multiplicative norm always exists by \cite[Remark 2.8.18]{kedlaya2015Relativea} in conjunction with \cite{dine2022Topologicala}, Remark 2.13). 
    Then we have
    \begin{equation*}
        I^\sharp = \bigcap_{\substack{\phi \in \mathcal{M}(R^{\flat}[1/\pi^{\flat}]) \\ \ker(\phi) \supseteq I}}\ker(\phi^\sharp) \cap R.
    \end{equation*}
\end{proposition}

\begin{proof}
    By \cite[Proposition 4.9]{dine2022Topologicala},
    for any spectrally reduced ideal \(\mcalI\) of \((R[1/\pi])^\flat\),
    the map
    \begin{equation*}
        \mcalI \mapsto \bigcap_{\substack{\phi\in\mathcal{M}(R^{\flat}[1/\pi^{\flat}]) \\ \ker(\phi) \supseteq \mcalI}} \ker(\phi^\sharp)
    \end{equation*}
    is inverse to the map $\mcalJ \mapsto \mcalJ^{\flat}$ from the set of spectrally reduced ideals of $R[1/\pi]$ to the set of spectrally reduced ideals (or, equivalently, closed radical ideals, see \cite[Remark 2.21]{dine2022Topologicala}) of $(R[1/\pi])^{\flat}=R^{\flat}[1/\pi^{\flat}]$.
    But for an ideal $I$ of $R^{\flat}$ and $\phi \in \mathcal{M}(R^{\flat}[1/\pi^{\flat}])$,
    we have $\ker(\phi) \supseteq I[1/\pi^{\flat}]$ if and only if $\ker(\phi) \supseteq I$, and by \Cref{CorrespSpecRed} a radical ideal $J$ of $R$ is perfectoid if and only if the ideal $J[1/\pi]$ of $R[1/\pi]$ is spectrally reduced.

    Consequently, 
    \begin{equation*}
        I \mapsto \bigcap_{\substack{\phi \in \mathcal{M}(R^{\flat}[1/\pi^{\flat}]) \\ \ker(\phi) \supseteq I}}\ker(\phi^\sharp)\cap R
    \end{equation*}
    is inverse to the map $J \mapsto J^{\flat}$ from the set of perfectoid ideals of $R$ to the set of $p^{\flat}$-adically closed radical ideals of $R^{\flat}$:
    Indeed, the tilt of the right hand side is equal to the inverse image of \(I[1/\pi^\flat]\) via \(R^\flat \to \overline{R^\flat}[1/\pi^\flat]\) by the above arguments and the inverse image is \(I\) because \(R^\flat/I\) is \(\pi^\flat\)-torsion-free.
    By \Cref{CorrespPerfdIdeals}, the map $I \mapsto I^\sharp$ is also inverse to $J \mapsto J^{\flat}$. 
    The assertion follows.
\end{proof}

\begin{remark}
    Note that the space $\Spec_{\perfd}(R^{\flat})$ of \(p^\flat\)-adically closed (or, by \Cref{Perfectoid elements and untilting}, of $\pi^{\flat}$-adically closed) prime ideals of \(R^\flat\) can be written as
    \begin{equation*}
        \Spec_{\Top}(R^\flat[1/\pi^\flat]) \cup \Spec(R^\flat/(\pi^\flat))=\Spec_{\Top}(R^\flat[1/\pi^\flat]) \cup \Spec(R/(\pi)).
    \end{equation*}
    If we restrict the tilting map \((-)^\flat\) to \(\Spec_{\Top}(R[1/\pi]) = D_R(\pi) \cap \Spec_{\perfd}(R)\),
    the homeomorphism induces a homeomorphism of topological spectra \(\Spec_{\Top}(R[1/\pi]) \cong \Spec_{\Top}(R^\flat[1/\pi^\flat]) \cong \Spec_{\Top}((R[1/\pi])^\flat)\) by \Cref{EqualitySomeSpectrum}.
    
    Moreover, by \Cref{TwoUntiltingMaps}, the analogous restriction of the inverse map $(-)^\sharp$ is the same as the map $(-)^\sharp$ defined in \cite[Section 4]{dine2022Topologicala}, and thus \Cref{CorrespPrimePerfdIdeals} provides a new proof of the main result of \cite{dine2022Topologicala} (loc.~cit., Theorem 4.16): For a perfectoid Tate ring $A$ of characteristic $0$, choose a topologically nilpotent unit $\pi$ of $A$ such that $\pi^{p}$ divides $p$ in $A^{\circ}$ and then apply the above arguments to the perfectoid ring $R=A^{\circ}$ and its perfectoid element $\pi\in A^{\circ}$.
    We note that this new proof of loc.~cit., Theorem 4.16, is more algebraic in nature in that, unlike the previous proof, it does not make use of the homeomorphism of Berkovich spectra 
    \begin{equation*}
        \mathcal{M}(R[1/\pi]) \simeq \mathcal{M}(R^{\flat}[1/\pi^{\flat}])
    \end{equation*}
    from \cite[Theorem 3.3.7(c)]{kedlaya2015Relativea}, or the homeomorphism of adic spectra 
    \begin{equation*}
        \Spa(R[1/\pi], R[1/\pi]^{\circ}) \simeq \Spa(R^{\flat}[1/\pi^{\flat}], R^{\flat}[1/\pi^{\flat}]^{\circ})
    \end{equation*}
    from \cite[Corollary 6.7(iii)]{scholze2012Perfectoida}. 

    
\end{remark}


\begin{remark}[{cf. \cite{kaplansky1947Topological} and \cite{dine2022Topologicala}}] \label{CounterexamplePerfdIdeal}
    There exist perfectoid rings $R$ of characteristic $0$ such that $\Spec_{\perfd}(R)$ is not equal to $\Spec(R)$ and $\Spec_{\perfd}(R^{\flat})$ is not equal to $\Spec(R^{\flat})$. An example is as follows.

    Let \(K\) be a perfectoid field with pseudo-uniformizer \(\pi\) and let \(X\) be an infinite totally disconnected compact Hausdorff space.
    Consider the Banach ring \(C(X, K)\) of continuous \(K\)-valued functions on \(X\) whose topology is defined by uniform convergence on all of \(X\).
    Then \(C(X, K)\) is a perfectoid Tate ring whose pseudo-uniformizer is the constant function \(x \mapsto \pi\) which we also denote \(\pi\) and \(C(X, K)^\circ=C(X, K^\circ)\) is a perfectoid ring.
    Moreover, any closed prime ideal of \(C(X, K)\) is maximal by \cite[Theorem 29]{kaplansky1947Topological}
    but \(C(X, K)\) has a prime ideal \(\mfrakp\) which is not a maximal one by the example after the proof of \cite[Theorem 30]{kaplansky1947Topological}.
    Thus \(\mfrakp\) is a non-closed ideal of \(C(X, K)\), in which case
    \begin{equation*}
        \mfrakp^\circ \defeq \mfrakp \cap C(X, K^\circ)
    \end{equation*}
    is a prime ideal of the perfectoid ring \(C(X, K^\circ)\) which is not closed with respect to the subspace topology induced from \(C(X, K^\circ)\) (if \(\mfrakp^\circ\) was closed, then the quotient \(C(X, K^\circ)/\mfrakp^\circ\) would be \(\pi\)-adically separated and then the quotient Tate ring \(C(X, K)/\mfrakp=(C(X, K^\circ)/\mfrakp^\circ)[1/\pi]\) would be Hausdorff, which is a contradiction).

    If the perfectoid field \(K\) is of characteristic \(0\) and contains \(\setQ_p\),
    we can choose \(p\) as \(\pi\) and then \(C(X, K^\circ)\) is equipped with the \(p\)-adic topology.
    This shows that \(C(X, K^\circ)/\mfrakp^\circ\) is not \(p\)-adically complete and, in particular, \(\mfrakp^\circ\) is not perfectoid.

    Similarly, in the above situation we can let \(\mfrakq\) be a non-closed prime ideal in the perfect(oid) Tate ring \(C(X, K^\flat)\) and then \(\mfrakq^\circ \defeq \mfrakq \cap C(X, K^{\flat \circ})\) is a prime ideal in the perfect(oid) ring \(C(X, K^{\flat \circ})\) of characteristic \(p\) which is not \(p^\flat\)-adically closed.
\end{remark}

As an application, we show that a perfectoid ring (not a perfectoid Tate ring) is an integral domain if and only if its tilt is an integral domain.
This contains (and is equivalent to) a previous result about perfectoid Tate rings (\cite[Corollary 4.17]{dine2022Topologicala}).

\begin{corollary} \label{PerfdDomain}
    Let $R$ be a perfectoid ring of characteristic $0$, let $\pi$ be a perfectoid element in $R$ and let $\pi^{\flat}\in R^{\flat}$ be such that $\pi^{\flat\sharp}$ is a unit multiple of $\pi$ in $R$.
    Then $R/R[\pi^{\infty}]$ is an integral domain if and only if $R^{\flat}/R^{\flat}[(\pi^{\flat})^\infty]$ is an integral domain. 
    In particular, \(R\) is an integral domain if and only if its tilt \(R^\flat\) is an integral domain.
\end{corollary}

\begin{proof}
    By \cite[2.1.3]{cesnavicius2023Purity}, the $\pi^{\flat}$-torsion-free quotient $R^{\flat}/R^{\flat}[(\pi^{\flat})^\infty]$ is the tilt of the perfectoid ring $R/R[\pi^{\infty}]$.
    Hence it suffices to prove the assertion in the case when $R$ is $\pi$-torsion-free (and thus $R^{\flat}$ is $\pi^{\flat}$-torsion-free). 
    Since the zero ideal of $R$ is a perfectoid ideal, the assertion follows from \Cref{CorrespPrimePerfdIdeals}. 
\end{proof}

\begin{remark}
    Actually, the above corollary is a direct consequence of the previous result \cite[Corollary 4.17]{dine2022Topologicala}.
    In fact, as we saw in the proof of \Cref{PerfdDomain}, it suffices to prove the corollary in the case when $R$ is $\pi$-torsion-free (and $R^{\flat}$ is $\pi^{\flat}$-torsion-free).
    In this case, we have injective maps $R\hookrightarrow R[1/\pi]$ and $R^{\flat}\hookrightarrow R^{\flat}[1/\pi^{\flat}]$ and we can apply the previous result to conclude.
    However, our proof does not rely on the results of \cite{dine2022Topologicala}.
    In particular, we obtain a new proof of loc.~cit., Corollary 4.17, which is more algebraic in nature in the sense that it does not use homeomorphisms between Berkovich spectra or adic spectra.

\end{remark}




\renewbibmacro{in:}{}
\printbibliography 

\end{document}